\documentclass[a4paper,12pt]{amsart}

\usepackage{fullpage}
\usepackage{amsmath}
\usepackage{amssymb}
\usepackage{float}
\usepackage{colonequals}
\usepackage{graphicx}

\usepackage{bbm}
\newcommand{\calO}{{\mathcal O}}

\DeclareMathOperator*{\res}{res}
\DeclareMathOperator{\Cl}{Cl}
\DeclareMathOperator{\Nm}{Nm}
\DeclareMathOperator{\ord}{ord}
\DeclareMathOperator{\Reg}{Reg}
\DeclareMathOperator{\hnum}{h}

\def \magma {Magma}
\def \Magma {\magma}

\newcommand{\abs}[1]{\lvert #1 \rvert}
\newcommand{\defi}{\textsf}

\newcommand{\bbQ}{{\mathbbm Q}}
\newcommand{\bbZ}{{\mathbbm Z}}
\newcommand{\Q}{{\bbQ}}

\newcommand{\fraka}{{\mathfrak a}}
\newcommand{\frakp}{{\mathfrak p}}
\DeclareMathOperator*{\disc}{disc}

\renewcommand{\pmod}[1]{\nobreak\ifinner\mkern8mu\else\mkern8mu\fi
(\textup{mod}\,\,#1)}

\numberwithin{equation}{section}

\theoremstyle{plain}

\newtheorem{thm}[equation]{Theorem}

\theoremstyle{remark}

\newtheorem{rmk}[equation]{Remark}
\newtheorem{ques}[equation]{Question}
\newtheorem{example}[equation]{Example}

\theoremstyle{definition}

\usepackage[backref=page,breaklinks]{hyperref}
\hypersetup{colorlinks=true,urlcolor=blue,citecolor=blue,linkcolor=blue}

\usepackage{thmtools}
\usepackage[nameinlink]{cleveref}
\begin{document}

\title[Magma class groups and unit groups]{Computing class groups and unit groups \\ in Magma}

\author{Andreas-Stephan Elsenhans}

\address{School of Mathematics and Statistics\\ University of Sydney, NSW, 2006\\ Australia}
\address{Institut f\"ur Mathematik\\ Universit\"at W\"urzburg\\ Emil-Fischer-Stra\ss e 30\\ D-97074 W\"urzburg\\ German
y}
\email{stephan.elsenhans@mathematik.uni-wuerzburg.de}
\urladdr{https://www.mathematik.uni-wuerzburg.de/institut/personal/elsenhans.html}

\author{John Voight}

\address{School of Mathematics and Statistics\\ University of Sydney, NSW, 2006\\ Australia}
\email{jvoight@gmail.com}
\urladdr{https://jvoight.github.io/}


\begin{abstract}
We describe the computation of class groups and unit groups of number fields as implemented in Magma (V2.29).  After quickly reviewing the main algorithms based on factor bases, relation collection, and analytic class number evaluation, we distinguish their behavior across formalizable, rigorous, GRH-conditional, and heuristic regimes. 
\end{abstract}

\maketitle

\theoremstyle{definition}
\newtheorem*{ttt}{}

\section{Introduction}

\subsection*{A bit of motivation}

In computational algebraic number theory, determining the class group and unit group of the ring of integers of a number field is a fundamental algorithmic task. Their computation remains of significant theoretical interest---many central questions in number theory concern properties of these groups, including their distribution---and they are indispensible in practice, underpinning explicit methods across arithmetic geometry.  We might even consider these algorithms as one key benchmark for a computational algebra system used in number theory.

\subsection*{Takeaways}

We consider five possible regimes of computing. 

\begin{itemize}
\item \emph{Formalizable}: rigorously proven in a manner that could be verified by a formalized proof assistant---in particular, unconditional.  If numerical calculations are performed, they must use ball arithmetic.    
\item \emph{Rigorous}: proven, but we allow numerical calculations if they include floating-point error analysis (but can ignore any possible accumulated round off errors).  
\item \emph{GRH-conditional formalizable} and \emph{GRH-conditional rigorous}: same but we assume the Generalized Riemann Hypothesis (GRH).
\item \emph{Heuristic}: good heuristics may be assumed, so it is very likely to be the correct answer, but it should not be considered proven.  In particular, we allow good, reliable, nonrigorous numerical computations. 
\end{itemize}

One step beyond formalizable would be \emph{formalized}, where steps along the way are output with the intention of entering them into a formalized proof assistant.  

In this article, we document class group and unit group algorithms in Magma \cite{Magma}.  With the above in mind, the takeaways are as follows.

\begin{enumerate}
\item By default (for example, calling \texttt{ClassGroup} with no optional arguments), the class group algorithm in Magma is formalizable. A ``standard'' approach is taken: generators are taken up to the Minkowski bound, relations are found, and then the class group is confirmed to be saturated.  The latter steps are done using exact computations. 

The bad news: a bug was uncovered (in fact, it has been a problem for a some time!)\ which turned off saturation of the class group. So the results returned were not formalizable. This has been fixed as of V2.29; and, for what it is worth, by the way that the relations are computed it is quite unlikely that a wrong answer was returned.  

\item By default, the unit group algorithm in Magma is rigorous.  A lower bound of the regulator is computed and the class group is $p$-saturated with respect to all primes $p$ smaller than the quotient of the regulator and the regulator bound. As all regulator computations use floating point numbers, this is not formalizable; but it could be made so, with ball arithmetic.

\item When the GRH flag is added, in V2.28 (indeed going back many versions), a non-rigorous evaluation of the Euler product was used and so the class group and unit group were not rigorous or formalizable. In V2.29, this bug is fixed.

Calling {\tt ClassGroup} with proof level {\tt GRH} will result in a GRH-conditional rigorous computation. Further, the proof level {\tt Subgroup} in combination with a GRH-bound for the generators of the class group will result in a GRH-conditional formalizable algorithm,
as the check for saturation is performed.

Calling {\tt UnitGroup} with the additional parameter {\tt GRH} will skip the saturation and therefore result in a GRH-conditional rigorous computation. Without that parameter, the saturation check is done as well and the result will be rigorous.

\item As it turns out, the bit of good news is that a non-rigorous Euler product evaluation turns out to be extremely close and very useful as a heuristic---so much so that again it seems extremely unlikely that a wrong answer was ever returned.  We discuss at the end some possible heuristic algorithms that follow up on this observation.
\end{enumerate}

\subsection*{Contents}

This article gives a detailed account of how class groups and unit groups of number fields are computed in Magma V2.29, peppered with observations and questions throughout.  We start in \Cref{sec:history} with a brief history.  \Cref{sec:basic} recalls the basic definitions and fundamental algorithmic problems, highlighting issues such as representation of number fields, maximal orders, and the Picard group. In \cref{sec:factorbase}, we review the core factor base method, describing how generators and relations are obtained under various assumptions, how to decide when enough relations have been found, and how this connects to heuristics and conjectures.  Then we turn to analytic techniques in \cref{sec:analytic}, detailing both rigorous GRH-conditional estimates and heuristic Euler product evaluations of the Dedekind zeta residue, with discussion of error analysis and stopping criteria.  \Cref{sec:unitgroup} focuses on unit group algorithms, including classical methods for quadratic fields, strategies for detecting and saturating missing units, and the use of discrete logarithms modulo primes.  Then in \cref{sec:finishing} we explain how the class group is saturated and verified, describing the detection of missing generators or relations and the structure of the saturation algorithm.  We conclude in \cref{sec:conclude} with a summary with benchmarks, along with observations on typical error patterns and some possible future directions.

\subsection*{Acknowledgements}

We would like to thank Eran Assaf, Edgar Costa, Tim Dokchitser, Claus Fieker, Nicole Sutherland, Allan Steel, and Alain Chavarri Villarello for helpful conversations.  Thanks also to the participants in the \emph{Magma Meeting: Rational Points 2025} for their input including some suggestions for future work in the final section.

\section{A bit of history} \label{sec:history}

% Here is what I can get from the handbook. The facilities for general number fields in Magma are based on the KANT V4 package developed by Michael Pohst and collaborators, first at D\"usseldorf and then at TU Berlin. This package provides extensive machinery for computing with maximal orders of number fields and their ideals, and computing the class group and unit group.

% \begin{itemize}
% \item Jean-Fran\c{c}ois Biasse implemented a quadratic sieve for computing the class group of a quadratic field. He also developed a generalisation of the sieve for number fields of degree greater than 2.
% \item Steve Donnelly provided a new implementation of the general class group algorithm.
% \item Claus Fieker built on the class group package by providing explicit algorithmic class field theory.
% \item J\"urgen Kl\"uners and Sebastian Pauli developed algorithms for computing the Picard group of non-maximal orders and for embedding the unit group of non-maximal orders into the unit group of the field.
% \item The fast algorithm of Wieb Bosma and Peter Stevenhagen for computing the 2-part of the ideal class group of a quadratic field has been implemented by Mark Watkins.
% \end{itemize}
% It would be interesting to check this with subversion and confirm, perhaps seeing if others contributed.

% The authors of the handbook on number fields and orders are listed as Claus Fieker, Wieb Bosma, and Nicole Sutherland.

Many individuals have contributed to the current package in Magma.  We would appreciate hearing of any corrections or additions to our attempt to reconstruct the history here.  Since they often go together, we lump together code to compute the class group and unit group together, referring only to the former.

The class group code in Magma originates from code taken from the KANT package in 1994. This code was written in the early 1990s by Johannes Graf von Schmettow, then rewritten chiefly by Klaus Wildanger and Florian Hess in 1996 and then updated in Magma.
% [C code in src/Classgroup; this remained the main code for general number fields until the Donnelly version came in 2012 and became the new default].

In the mid 1990s, code from Pari was installed and modified by Alexandra Flynn for class groups of quadratic fields. Although it was heavily used and supported during that time, it is no longer called (except through code for binary quadratic forms). 

Claus Fieker built on the class group package by providing explicit algorithmic class field theory during the period 2000--2011~\cite{F01}.

J\"urgen Kl\"uners and Sebastian Pauli developed algorithms for computing the Picard group of non-maximal orders and for embedding the unit group of non-maximal orders into the unit group of the field in the mid 2000s. Also during this period, Mark Watkins implemented the fast algorithm of Wieb Bosma and Peter Stevenhagen for computing the $2$-part of the ideal class group of a quadratic field. % [2006; code in package/Geometry/CrvCon/bosma.m]

In 2011, Jean-Fran\c{c}ois Biasse implemented a quadratic sieve for computing the class group of a quadratic field in collaboration with Steve Donnelly and Nicole Sutherland. He also developed a generalisation of the sieve for number fields of degree greater than 2; it is reported to perform well for degrees up to $5$ and to some extent in degree $6$. 

In the period 2012--2014, Steve Donnelly provided a new implementation of the general class group algorithm using a random walk on ideals with heavy use of LLL on ideal bases to find smooth relations---this is generally much faster than the old KANT code. Allan Steel also made significant improvements in this code. 

Finally, Allan Steel implemented a distributed parallel version of the class group code in 2022.

\section{Basic algorithms} \label{sec:basic}

We refer to the \Magma\ handbook and the article by Claus Fieker~\cite{F06} for an overview. For general references, there is also an article by Lenstra~\cite{L92} and the books of Cohen~\cite{Cohen,Cohen2}. Many of the results were originally obtained by Zassenhaus, Post, Cohen, Buchmann 
% (order by year of birth)
and many others, but this is not the right place to attempt to give detailed references.

\subsection*{Setup}

In \Magma, a number field $K=\Q(\alpha)$ is represented by the minimal polynomial $f(x) \in \Q[x]$ of a primitive element $\alpha$. (In fact, it is represented in KANT by an equation order, which is why there are sometimes issues when the polynomial does not have integer coefficients.) 

\begin{ques} 
We consider here only the absolute case.  However, much of what is done can be made relative.  This could be important if considering many extensions $K$ of a given field $F$ (to avoid recomputing data coming from the base field). 

For example, what speed improvements could we hope for in the case of CM extensions $K$ of a totally real field $F$ (analogous to the case of imaginary quadratic fields)?  As the rank of the unit group of a CM extension coincides with the unit rank over the totally real field, a specialized unit group algorithm may be
helpful.
\end{ques}

\begin{rmk}
We could also work more broadly with \'etale algebras; this is explained and implemented in Magma by Stefano Marseglia. However, these algebras project onto their number field factors and a basic primitive used in these algorithms is the number field case. 
\end{rmk}

\begin{ques}
Do we gain any efficiency in working with number fields embedded in higher codimension? For example, $K=\Q[x,y]/(f(x,y),g(x,y))$? If a lot of time is spent mapping from a power basis, then perhaps. If everything is computed relative to a small (LLL-reduced?)\ integral basis, then this should be independent of the way that the algebra is represented. Perhaps we should lead with this way of presenting arithmetic in number fields---sure, the input is a minimal polynomial but most of the action is in finding a good basis for orders?
\end{ques}

We suppose that a maximal order has been computed. 

\begin{ques}
The assumption that the order is maximal is unnecessary and could be quite expensive to prove. The algorithms work perfectly well without this assumption, working instead with the Picard group and when something goes wrong, we typically have produced a superorder.  Is this difficult to implement? 
\end{ques}

\subsection*{Fundamental algorithms}

We write $\Cl \calO_K$ for the class group of $\calO_K$, and write $\hnum(\calO_K)$ for its order. Let $r_1,r_2$ be the number of real, complex places of $K$ and let $d \colonequals \disc(\calO_K)$ be the discriminant. 

The class group is a finite abelian group by the geometry of numbers, and the unit group is a finitely generated abelian group by Dirichlet's unit theorem, of rank $r_1+r_2-1$. Applying logarithms to the absolute values under the set of $r$ archimedean places of $K$ maps the unit group to a lattice of rank $r - 1$. The kernel of this map are the roots of unity in $\calO_K^\times$. The covolume of this lattice is the regulator.

Given as input the maximal order $\calO_K$, there are two fundamental (and related) problems. 
\begin{enumerate}
\item Class group: give as output a computable isomorphism between an abstract finite abelian group and the class group $\Cl \calO_K$ (as a group of classes of ideals)---so one can evaluate both the map (give an ideal which represents the given class) and its inverse (given an ideal, find its class). 
\item Unit group: give as output a computable isomorphism between an abstract finitely generated abelian group and the unit group $\calO_K^\times$. 
\end{enumerate}

We allow the output of the unit group to be in a compact representation, as a product of elements with exponents. Calling {\tt UnitGroup} with the optional parameter {\tt Raw} gives the user
access to this representation.

We can easily get roots of unity in the field, so we ignore this throughout. One way to obatain them is a search for elements of small $T_2$-norm, which is part of the saturation. But, usually they are obtained during the unit group computation.

The example 
\begin{verbatim}
> _<x> := PolynomialRing(Rationals());
> K := NumberField(x^2 + 890232348011);
> ord := MaximalOrder(K);
> time c1, c2 := ClassGroup(ord : Proof := "GRH");
Time: 0.070
> 1.0 * Norm(c2(30000*c1.1));
6.02233400649697968550171564231E172002
\end{verbatim}
shows that the class group map can result in huge representatives, if the class group is a large cyclic group. 

\section{The factor base method: an overview} \label{sec:factorbase}

\subsection*{Basic idea}

What better way to compute a finitely generated abelian group than with generators and relations!  For the class group, we take a set of classes of prime ideals; we call this the \defi{factor base}. We then find relations, meaning elements $\alpha \in K^\times$ which factor completely over the factor base. We can combine relations until they have the trivial factorization, which gives a unit. 

Written out, let $S \colonequals \{\frakp_1,\dots,\frakp_m\}$ be a set of primes. Let
\[ \calO_{K,S}^\times \colonequals \{\alpha \in K^\times : \textup{$\ord_\frakp(\alpha) = 0$ if $\frakp \not\in S$} \} \]
be the subgroup of elements which factor over the primes in $S$, i.e., the $S$-unit group. Then there is an exact sequence
\begin{equation}
1 \to \calO_K^\times \to \calO_{K,S}^\times \xrightarrow{\phi} \bbZ^m \xrightarrow{\pi} \Cl \calO_K
\end{equation}
where the first map is the natural inclusion, the map $\phi$ is factorization
\begin{equation}
\begin{aligned}
\phi \colon \calO_{K,S}^\times &\to \bbZ^m \\
\alpha &\mapsto (\ord_{\frakp_i}(\alpha))_{i=1,\dots,m},
\end{aligned}
\end{equation}
and $\pi$ is taking the class 
\begin{equation}
\begin{aligned}
\pi \colon \bbZ^m &\rightarrow \Cl \calO_K \\
(e_1,\ldots,e_m) &\mapsto \prod_{i}^m [\frakp_i^{e_i}].
\end{aligned}
\end{equation}

The set $S$ generates the class group if and only if $\pi$ is surjective; in that case, if we can compute $\ker \pi$, we know the class group. So we search for enough elements in $\calO_{K,S}^\times$ to map to a generating system of $\ker \pi \leq \bbZ^m$. As a byproduct, we can find candidate generators for the unit group from $\ker \phi$. 

So over the next few subsections we will discuss generators, finding relations, and then how we know when we have found all relations. We break these up into the various regimes explained above. 

\subsection*{Formalizable generators}

The theorem of Minkowski provides that every ideal class of $\calO_K$ has an integral representative $\fraka$ with
\begin{equation} \label{eqn:nm}
\Nm(\fraka) \leq \frac{n!}{n^n} \left(\frac{4}{\pi}\right)^{r_2} \sqrt{\abs{d}}
\end{equation}
where $\Nm$ denotes the absolute norm. (This is one way to prove that the class group is finite, but it is not necessarily an efficient way to represent elements of the class group.)

So for something formalizable, to be sure we have generators we take all prime ideals up to the Minkowski bound. 

\begin{rmk}
Sometimes the Minkowski bound is shortened to $O(\sqrt{|d|})$. But, this is only literally true when the degree is fixed, as it does not take the leading coefficient into account.  Indeed, Stirling's approximation (which is quite good!)\ gives 
\[ \frac{n!}{n^n} \sim \frac{\sqrt{2\pi n}}{e^n} \approx 2.50 \sqrt{n} e^{-n} \]
which decays exponentially in $n$. Already for $n=10$, the constant is only $0.0012$. 

Although it is not meaningful to study the inequality for fixed discriminant, we might study number fields of small discriminant relative to the GRH lower bound, which is of the form 
\[ \sqrt{\abs{d}} \underset{\sim}{>} (60.84)^{r_1} (22.38)^{2r_2} \]
which suggests that a lower estimate for the Minkowski 
bound of the shape
\begin{equation*}  
 2.50 \sqrt{n} (22.38)^{r_1} (9.29)^{2r_2} 
\end{equation*}
is possible.
This expresses the Minkowski bound for number fields of large degree $n$ whose discriminant is small relative to the degree: it at least indicates that the bound is larger when there are more complex embeddings, which makes sense.
\end{rmk}

The Zimmert bounds improved the Minkowski constants; there is an elegant reformulation due to Oseterl\'e. The best bound is that for every class $C$, there is an integral ideal $\fraka$ either in $C$ or $\mathfrak{d} C$ where $\mathfrak{d}$ denotes the different of $\calO_K$ such that
\begin{equation} \label{eqn:zimmert}
\Nm(\fraka) \leq (0.128-o(1))^{r_1} (0.044-o(1))^{r_2} \sqrt{\abs{d}} 
\end{equation}
and
\begin{equation}
\Nm(\fraka) \leq (0.14-o(1))^{r_1} (0.051-o(1))^{r_2} \sqrt{\abs{d}} 
\end{equation}
for all ideal classes. Note how much better this is than Minkowski, which gives roughly
\begin{equation}
\Nm(\fraka) \underset{\sim}{<} 2.50 \sqrt{n} (0.36)^{r_1} (0.46)^{r_2} \sqrt{\abs{d}}. 
\end{equation}

It is of course no trouble to include the class of the different (for example, it is trivial for monogenic orders).

\begin{rmk}
The constants in Zimmert's inequality are known. Therefore, we have  {\tt ZimmertBound} in \Magma. But, the improvement is not as big as one might hope for, as the following example shows:
\begin{verbatim}
> _<x> := PolynomialRing(Rationals());
> f := &*[x-i : i in [-6..6]] + 1;
> OK := MaximalOrder(f);
> 1.0* Discriminant(OK);
1.62088452500733828715108860464E88
> MinkowskiBound(OK);
2617536668803912827212778710271533052902
> ZimmertBound(OK);
9496537377795252212801557901238143503
> 1.0 * $2 / $1;
275.630639323783598832733968433
\end{verbatim}
A constant factor, but in either case it is so large it is practically useless.
\end{rmk}

\begin{rmk}
There is no known unconditional bound for the norms of a set of \emph{generators} for the class group. The Minkowski bound allows us to represent every element of the class group this way, which is overkill.
\end{rmk}

\subsection*{Conditional generators}

Conditional on the GRH, we can improve on Minkowski as follows.

\begin{thm}[Bach] 
Assuming GRH, the set
$$
\{ \frakp \subset \calO_K : \Nm(\frakp) \leq 12 (\log \abs{d})^2 \} 
$$
generates $\Cl \calO_K$.
\end{thm}

\begin{proof}
This is a direct consequence of \cite[Theorem 4]{Bach}.
\end{proof}

Belabas--Diaz y Diaz--Friedman~\cite{MR2373197} gave another GRH bound.   In practice, one can use the minimum of the two.  This is implemented in \Magma\ via {\tt GRHBound} and is used by the class group algorithm.

\subsection*{Heuristic generators}

The Cohen--Lenstra--Martinet heuristics predict that the class group is ``close to cyclic''. More precisely, if $G$ is the Galois group of $K$ then the prime-to-($\#\mu_K\#G$) part of the class group has a matrix model in generators and relations like the above (with the size of matrices tending to infinity). For imaginary quadratic fields, the heuristic says that the odd part is cyclic with probability $97.7\%$. The odds only go up when considering higher degree fields, because they are modelled by a further quotient by units. 

Using the Chebotarev theorem as a heuristic, we expect prime ideals to land randomly as elements in this cyclic group, so again the probability is very high that one small ideal, surely a few, will generate.  
From this perspective, we do not need to take all primes up to a given bound. But then we still have to find relations, a topic we turn to next.

\begin{ques}
We also still have to worry about $\ell$-Sylow subgroups with $\ell \mid \#\mu_K\#G$.  For the $2$-part of quadratic fields, we have genus theory available, and recent heuristics are theorems in this case.  What does a generalization to arbitrary degree say heuristically?  
\end{ques}

\subsection*{Finding relations} 

By their construction as given above, the relations are elements of $\phi(\calO_{K,S}^\times) \leq \bbZ^m$. 

Many people have worked on different approaches to generate relations. 
Here, it is sufficient to know that we have some code that does it; and the longer we run it, the more relations we get.  We may want at least to be sure that we loop over small elements to find small units, as in \Cref{rmk:smallunits}.

In general, one can (and has to) view such code as a black box. First, some of the algorithms are randomized~\cite[section 6.5.2]{Cohen}. Second, if sieving is used to find relations, one might implement the optimizations listed in~\cite[section 10.4.3]{Cohen}. These result in a general speed up,
but the search for relations is no longer exhaustive in a well-defined search range.

\begin{ques}
What if we started by making relations, keeping those that satisfy some criterion, then form the generators based on the desired relations? This is the problem of finding a subset of $r$ rows which are supported on $r$ columns. Is this NP-complete? 
\end{ques}

Using linear algebra, given sufficiently many relations one can:
\begin{itemize}
\item
compute a sublattice of $\ker(\phi)$.
\item
compute a subgroup of the unit group.
\end{itemize}

A parallel version of the relation search was implemented by Allan Steel in 2020--2022. 
It is based on the farmer-worker-model and is turned on once \Magma{} is set to parallel mode via
{\tt SetNthreads}. As this allows to work with larger examples, the relation matrix is now represented
as a sparse matrix. 

\subsection*{When to stop searching?}

Given a method to find relations, we need a mechanism which indicates when to stop searching and proceed with simplifying the presentation.  

For the unit group: if we miss units, the unit group does not have the expected rank or the discriminant
of the unit lattice (regulator) is \emph{larger} than it is supposed to (fewer units means larger covolume) by a positive integer factor.

For the class group: if we miss relations, given that we always take a set of generators we have only a submodule of $\ker(\pi)$. Thus, the class number found is a \emph{multiple} of what it is supposed to be. 

\begin{rmk}
In the heuristic case, the factor base may also be too small and we may not have generators; but that is the point of the heuristic method, by design this is supposed to be very unlikely to happen.
\end{rmk}

In either case, because we are off by a positive integer factor, it would be enough to know the product $\hnum(K) \cdot \Reg(K)$ to a specified precision, so we turn to this now. 

\section{Analytic approaches} \label{sec:analytic}

The analytic class number formula reads:
\[ 
\res_{s = 1} \zeta_K(s) = 
\frac{2^{r_1} (2\pi)^{r_2} hR}{w \sqrt{\abs{d}}}
\] 
where $w=\#O_{K,\textup{tors}}$ is the number of roots of unity in $K$, $h=\hnum(\calO_K)$, and $R=\Reg(\calO_K)$. 

\subsection*{Formalizable evaluation}

We could try to compute the residue by using the general package for working with $L$-series, implemented by Tim Dokchitser~\cite{D04}. The given complexity is $O(\sqrt{\abs{d}})$, which makes it impractical unless the discriminant is small (in which case other methods are still faster). Also, the error terms are not worked out all the way to the end. In principle, this could be made formalizable using ball arithmetic. 

An explanation for this: the Euler product actually converges for the Dedekind zeta function (see below), but for elliptic curves at $s=1$ it does not---you need analytic continuation to push past $\mathrm{Re}(s)>3/2$. 

Currently, there is no formalizable method available in \Magma\ that uses an analytic approach (in particular, we do not have ball arithmetic). However, we still use the heuristic evaluation below to give a \emph{practical} guess as to when to stop looking for relations, as needed at the end of the previous section.  We still have to do something else to certify (see below), and for theoretical purposes we could skip this---but in practice, it gives us a good indication of when we should stop searching. 

\subsection*{Conditional evaluation}

We now explain a conditional, rigorous method. 

For $X>0$, define 
\begin{eqnarray*}
A_K(X) &\colonequals & \sum_{%\frakp, m, 
\frakp \subset \calO_K, N(\frakp^m) < X} 
\frac{\log N(\frakp)}{N(\frakp)^{m/2}} 
\left(\frac{\sqrt{X} \log(X)}{N(\frakp)^{m/2} \log(N(\frakp)^m)} - 1 \right) \\
f_K(X) &\colonequals & \frac{3(A_K(X) - A_\bbQ(X) + A_\bbQ(X/9)- A_K(X/9))}{2 \sqrt{X} \log(3X)}
\end{eqnarray*}

\begin{thm}[Belabas--Friedman, Bach] 
Assuming GRH, for all $X \geq 69$ we have
\[ \abs{\log \zeta_K^*(1) - f_K(X)} \leq \frac{2.324 \log\abs{d}}{\sqrt{X} \log(3X)} C \]
where
\[ C \colonequals 
\left(1 + \frac{3.88}{\log(X / 9)} \right)
\left(1 + \frac{2}{ \sqrt{\log\abs{d}}} \right)^2 +
\left( \frac{4.26 (n-1)}{\sqrt{X} \log \abs{d}} \right). \]
\end{thm}

\begin{proof}
See \cite[Theorem 1]{MR3266965}.
\end{proof}

To estimate the numerical error in the evaluation of 
$\log \res_{s=1} \zeta_K(s)$ we use $\epsilon$ for
the relative error of a single floating point operation.
Further, we will use the inequalities listed in~\cite[¬ß VII.1]{SMC06} 
and~\cite[¬ß VII.28]{SMC06}.

As a first step,  we estimate the sum of absolute values. For $A_K(X)$ this is bounded by
\begin{align*}
&\sum_{%\frakp, m, 
\frakp \subset \calO_K, N(\frakp^m) < X} 
\frac{\log N(\frakp)}{N(\frakp)^{m/2}} 
\left(\frac{\sqrt{X} \log(X)}{N(\frakp)^{m/2} \log(N(\frakp)^m)} \right) \\
\leq \, & [K : \bbQ] \sqrt{X} \log(X)
\left( \sum_{p < X} \frac{1}{p} 
+ \sum_{p, m \geq 2} \frac{1}{p^m}
\right) \\
\leq \, & [K : \bbQ] \sqrt{X} \log(X)
\left( \log \log X + 0.2615 + \frac{1}{2 (\log X )^2}
+ 0.77316 \right)\, .
\end{align*}
To simplify the notation, we set
$$
E(X) := \sqrt{X} \log(X)
\left( \log \log X + 0.2615 + \frac{1}{2 (\log X )^2}
+ 0.77316 \right)\, .
$$
Assuming that the computation of one term in $A_K$
uses $k$ floating point operations, we can bound the sum 
of all errors of all the terms of $A_K(X)$ by
$$
\epsilon k [K:\bbQ] E(X) \, .
$$
As the total sum has less than $[K : \bbQ] X$ terms, the error of summing all the terms in a pairwise summation scheme is bounded by
$$
\epsilon \frac{\log_2([K : \bbQ] X)}{1 - \epsilon \log_2([K : \bbQ] X)} [K : \bbQ] E(X) \, .
$$
This results in an error bound of
$$
\epsilon 2 (2 + \log \log X) ([K : \bbQ] + 1) \frac{3}{2} 
 \left(k + \log_2([K : \bbQ]  X) +  \log_2(X) \right) 
$$
for $\log \res_{s=1} \zeta_K(s)$.

If we use the standard data type {\tt single}, 
one gets $\epsilon = 2^{-24}$. This results in an error bound of 
less than $0.01$ for $X \leq 10^{10}$ and $[K:\bbQ] \leq 100$. 
We use the data type {\tt double}, in which case $\epsilon = 2^{-52}$;
therefore we end up with an error bound of less than $10^{-10}$  
for $X \leq 10^{10}$ and $[K:\bbQ] \leq 100$. 

To obtain a GRH-conditional rigorous algorithm, we compute the regulator to enough precision: estimates are worked out by Biasse--Fieker \cite[Lemma 4.4]{BF}. 

The current \Magma\ implementation for the GRH-conditional unit group
and class group with proof level {\tt GRH} rely on the correctness of the
residue. However, setting only the bound for the class group generators to {\tt GRH} 
will call the check of saturation and gives therefore a GRH-conditional formalizable
result.

\subsection*{Heuristic evaluation}

We can also just evaluate the Euler product ``directly''.

\begin{thm} 
The infinite product 
$$
\zeta_K^*(1) \colonequals \res_{s=1} \zeta_K(s) =
\lim_{s \rightarrow 1} \frac{\zeta_K(s)}{\zeta(s)}
= \prod_p \frac{\zeta_{K,p}(1)}{\zeta_p(1)} = 
\prod_{p} \frac{1 - 1/p}{\prod_{\frakp \mid p} (1-1/\Nm(\frakp))} 
$$ 
over primes $p$ converges. 
\end{thm}

\begin{proof}
The explicit bounds given in~\cite[Theorem 1]{Garcia} imply the convergence
for any number field $K$. 
\end{proof}

\begin{example}
We get
$$
\left|\res_{s=1} \zeta_K(s) - 
\prod_{p < X} \frac{1 - 1/p}{\prod_{\frakp \mid p} 1 - 1/N \frakp} \right| < 0.1461
$$
for $X = \exp(10^{34})$ and $K = \bbQ(i)$. 
\end{example}

\begin{rmk}
For the approximation
$$
A(X) := \prod_{p <X} \frac{1-1/p}{\prod_{\frakp \mid p, N\frakp < X} 1 - 1 /N \frakp},
$$
the GRH based estimate
$$
| \log(\res_{s=1} \zeta_K(s)) - A(X)| < \frac{2 \log |d| + (0.928 n + 0.754) \log(X)}{\sqrt{X}}
\mbox{ for } X > 1000
$$
was worked out in~\cite[Thm.~2, Tab.~2]{bach_94}. 
\end{rmk}
%An explicit estimate for the speed of convergence is known. Give reference? 
%However, the resulting bounds are too large to be computed. Give an example? 

Magma (V2.28) computes the finite product 
\[ \res_{s=1} \zeta_K(s) \approx \prod_{p < 1000} \frac{\zeta_{K,p}(1)}{\zeta_p(1)}. \]
If this does not match the class and unit group found with an error of 5\%, the bound is doubled, and this is repeated recursively. 
A test example with extra verbose printing is
\begin{verbatim}
> _<x> := PolynomialRing(Rationals());
> SetKantVerbose("CLASS_GROUP_CHECK",1);
> f := t^16 - 36*t^14 + 488*t^12 - 3186*t^10 + 10920*t^8 
 - 19804*t^6 + 17801*t^4 - 6264*t^2 + 64;
> OK := MaximalOrder(f);
> ClassGroup(ord : Proof := "GRH");
Euler-Product bound: 1000 
Euler-Product bound: 2000 
Abelian Group of order 1
\end{verbatim}
(Alas, \texttt{SetKantVerbose} only works on the development version.) Thus, the increase in factors works. 
Note that to \emph{ensure} 5\% error using the above GRH-based estimate would require a bound of $10^6$ even in small examples.

The above example is the most extreme one within more than 1000 polynomials. Here, the relative error of the residue with the prime bound 1000 found is below 7.5\%. If we consider the partial evaluation 
\[ \frac{\prod_{p<1000} \zeta_{K,p}(1)/\zeta_p(1)}{\res_{s=1} \zeta_K(s)} \]
assuming GRH, we find in this experiment:
\begin{itemize}
\item mean value $1.012$, 
\item standard deviation $0.015$, 
\item skewness $-0.0166$, and 
\item kurtosis $3.114$. 
\end{itemize}
Because of the 5\% error heuristic, the error has to be about 30 standard deviations to cause trouble. 

\begin{example}
We picked the field given by the randomly chosen polynomial
$x^8 + 9127 x^3 + 9127 x^2 + 9127 x + 18254$ for a closer inspection. 
It results in a maximal order with a 45-digit discriminant. 
We applied both methods described above to approximate the residue. 
The plot below shows the resulting approximation errors of the logarithm 
of the residue and the error
estimates as a function 
of the used prime bound. 

\begin{center}
\includegraphics[scale=0.65]{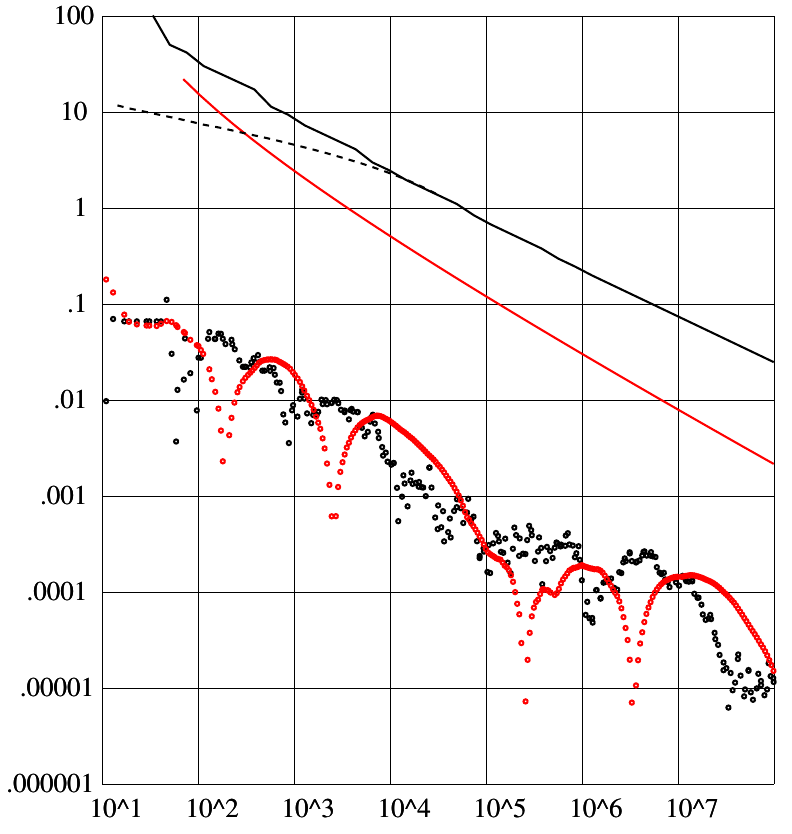}
\end{center}

The color black is used for the Euler product approach and Bach's estimate. The red color is used for the weighted sum estimate and the Belabas-Friedman error bound.
Each dot shows the maximal error that occures in the interval that it covers whereas the solid lines show the bounds. 
The dashed line shows a naive improvement of Bach's bound for small $p$. 
We can read off the plot that in practice both methods give about the same error, whereas the error estimates show a much bigger difference.
Note that the downward spikes in the red dots indicate a change of sign in the error, whereas the error of the Euler product oscillates at a much higher frequence. Therefore, the changes of sign are not visible in his type of plot.

The next plot compares the two errors
with the standard deviation of random
Euler products.

\begin{center}
\includegraphics[scale=0.65]{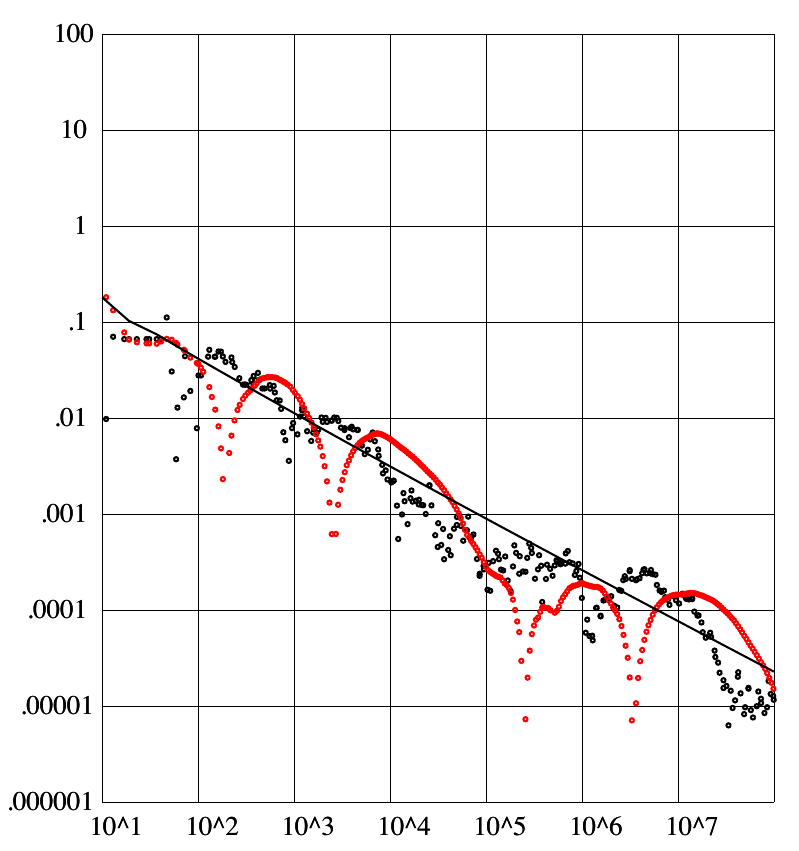}
\end{center}

The plot indicates that, a stochastic model gives an accurate representation of the errors. 
The article~\cite{GS03} may be interesting for a more theoretic perspective on this.
\end{example}

\begin{ques}
Does this hold up for imaginary quadratic fields of large discriminant? How should $1000$ actually scale with the parameters (degree, discriminant)?
% Further examples are: 
% EulerProduct(MaximalOrder(QuadraticField(36024103098795)),1000);
% 1.01416350958861737372781049533
% > ResidueGRH(MaximalOrder(QuadraticField(36024103098795)),10000000.0);
% 0.9488723368079085
%
% ord := MaximalOrder(QuadraticField(-74703619898129968610855258832360617333830294300803334449758241851));
%  ResidueGRH(ord,10000000.0);
% 1.091118639457539
% > EulerProduct(ord,1000);
% 1.15599481346829525480057368073
% > $1 / $2;
% 1.059458405039263
%
\end{ques}

\begin{rmk}
The GRH is a perfectly good heuristic, and some helpful GRH estimates are given by Duke \cite[Proposition 5]{Duke}. For example, we could add the assumption that the ratio $\zeta_K(s)/\zeta(s)$ is an entire Artin $L$-function---is this an additional assumption we should add to the hierarchy from the introduction? 
%Can we say anything more about the $O(1)$?     
\end{rmk}

In any case, this finishes the description of a heuristic algorithm for computing the class group and unit group. 

\section{Unit group algorithms} \label{sec:unitgroup}

We now turn to the formalizable, rigorous, conditional, and heuristic algorithms focused on the unit group. 

\subsection*{Quadratic fields} 

Imaginary quadratic fields have a finite unit group.  The unit group of a real quadratic field can be determined by the continued fraction algorithm in quasi-linear time in the size of the output (which can be exponential in the input size). Another algorithm to compute a unit is to use Stark's method and evaluate the $L$-series. 

\begin{ques}
A quasi linear version of the continued fraction algorithm is implemented in magma
{\tt /package/Ring/FldQuad/fund\_unit.m}.
When not using a power product representation of units, this approach compares reasonably favorably in speed
with the factor base approach. 
%The run time of the factor-base method is dominated by converting the unit
%from its power product representation to the standard presentation.
%Are either of these algorithms implemented in Magma? 
%Are they ever more efficient than an adapted factor-base method? 
\end{ques}

\subsection*{Missing units} 

At this stage in the algorithm, we may suppose we have the roots of unity and (by continuing to generate relations) that we have a set of multiplicatively-independent units of rank equal to $r_1+r_2-1$. 

Then any missing unit would have an $n$th power in our given subgroup for some $n\geq 2$, so we need to \emph{saturate} the lattice of units. Our first task is to give an upper bound for $n$. 

Any unit outside the subgroup generated would have to have bounded logarithmic height (thinking of the lattice in the log Minkowski embedding), hence bounded exponential height (by the arithmetic-geometric mean inequality):
\[ \log \Nm(x) \leq \log \prod_i (\abs{x}_i)^2 = (\sum_i \log \abs{x}_i)^2 \leq \sum_i (\log \abs{x}_i)^2 \leq \lambda_1(\log \calO_K^\times). \]
An approach \emph{not} taken by \Magma\ is to enumerate all elements up to that height. 

\begin{rmk} \label{rmk:smallunits}
In principle, this could be turned into a rigorous algorithm, at least once we have found (the roots of unity and) a set of linearly independent units. 
\end{rmk}

Instead, we bound $n$ as follows. We enumerate all elements that have at most twice the length of the last entry in the LLL bases of the given order.
% (function {\tt reglb\_C\_calc}).
This gives lower bounds for the logarithmic norm of generators of the unit group. By Minkowski's
inequalities this gives a lower bound on the regulator.
% Details are in the function {\tt order\_reg\_lbound\_calc}. This part of KANT is in
% \centerline{\tt /src/Units/reg\_lbound\_calc.c .} 
As this function does an enumeration of small elements, it also determines all the roots of unity contained in the maximal order.

Using this, the index is bounded from above by the quotient 
$$
\frac{\mbox{regulator of known subgroup}}{\mbox{lower bound of regulator}}\, .
$$
The check for saturation is performed for all prime indices up to
this bound.

\subsection*{Saturating the unit group}

We can now ensure that $\log \calO_K^\times$ is $\ell$-saturated for all primes $\ell \leq n$, with $n$ computed as in the previous subsection. 

In other words, given a prime $\ell$ we have to show that no $\ell$th root of a known unit gives a new unit. 

A possible approach, not taken by \Magma, is to just check if there is an $\ell$th root running over all possible elements in the lattice modulo $\ell$th powers. This will be expensive! 

Instead, we use the following reduction modulo $\frakp$ approach.
\begin{itemize}
\item
Take sufficiently many (say $m$) prime ideals $\frakp_i$ with $\ell \mid k_i^\times$.
Here, we set $k_i \colonequals (\calO / \frakp_i)$. 
\item
An $\ell$-th power of $\calO^\times$ reduces to an $\ell$-th power in $k_i^\times$.
\item
A $\ell$-th power is contained in an index $\ell^m$ subgroup of $\prod k_i^\times$.
\end{itemize}

This gives an algorithm, sketched as follows.
\begin{enumerate}
\item
Use discrete logarithms to map $k_i^\times / (k_i^\times)^\ell$ to an additive group.
\item
Use linear algebra in $(\bbZ / \ell \bbZ)^m$ to determine all the known units that 
reduce to the index $\ell^m$ subgroup.
\item
If a unit that is not an $\ell$-th power reduces to the index $\ell^m$ subgroup,
redo the above computation with a larger value of $m$ or terminate with an 
error message.
\item Return that the unit group is $\ell$-saturated.
\end{enumerate}
All the above is implemented in KANT and called by \Magma.

\begin{rmk}
Currently, the discrete logarithm in $k_i^\times / (k_i^\times)^\ell$ is 
computed via the discrete logarithm of $k_i^\times$. To ensure that this does not
get too hard, all the prime ideals $\frakp_i$ used are chosen to be degree one 
primes.

In the range this is used, Pohlig--Hellman reduction in combination with
a random walk algorithm for the discrete logarithm is a good choice. As this works by reducing to quotients of prime order, the 
code could be optimised, by computing only the discrete logarithm in the order $\ell$
quotient.

Generally, \Magma\ caches the prime ideals, but for the discrete logarithm no 
caching is visible in the code. 
\end{rmk}

%\begin{ques}
%Since discrete logarithms can be expensive, primes can and should be reused. Are they? 
%How are the primes chosen? Too small, and there will be collapsing. Too bad, and discrete %logarithms are hard. Do we look for primes where $\Nm(\frakp)-1$ is smooth? 
%\end{ques}

\begin{ques}
The Chebotarev density theorem is the heuristic for why this works. What does an optimistic version of this predict for the number of primes required? 
\end{ques}

% We couldn't get Magma to do the saturatio at all. The units are, as far
% as I could see, always computed using GRH Classgroup (which is, apart
% from the proof bit) probably optimal. But here they refused to
% saturate...

% > K<a> := NumberField(x^8-31);
% > UnitGroup(MaximalOrder(K));
% Entering order_reg_lbound_calc ...
% torsion unit is -1 (rank 2)
% lower regulator bound (reg_lbound_calc)
% 743.215281403143615069659181529949247633087151408354454159315018767615736320718\
% 8436557677954579681936915169555354088060227226152569631811078735707741291613302\
% 9031524693
% Entering order_units_indep_calc_strategy ...
% % strategy CLASS_GROUP
% Obtaining class group for
% x^8 - 31
% GRH bound = 165
% Starting core class group computation with bound 165
% Expected number of random elements per relation: 5
% Finished core class group computation, time: 0.210
% Class group proof phase
% Class group proof time: 0.000
% Entering order_reg_calc ...
% regulator
% 290386.45579927509589126780883311529934756276575174706147750047968909\
% 6506066535445798249732336050229695065117664401913041072145424643030186594184054\
% 70133766237907628660 (order_reg_calc)

% it should have saturated for primes up to 390, but it definitely spent
% time enumerating stuff for the lower bound...

% Makeing the field larger increases the time for enumeration, but I still
% have not seen saturation...

% All the best
%  Claus

\section{Finishing the class group}  \label{sec:finishing} 
 
In a similar way, we can saturate the class group.
More precisely, we need to detect the following two problems.  
\begin{itemize}
\item \textbf{Missing generators}.  If we miss a generator, then a prime ideal below the Minkowski bound gives a new ideal class.
\item \textbf{Missing relations}.  A missing relation results in a class group larger than correct. Further, the
decomposition of ideals into classes found is finer than the correct decomposition into ideal classes. 
As the result has a group structure, it will have elements of prime order 
that consist of principal ideals.
\end{itemize}

We do not have to take generators to start with, and this is the approach taken by \Magma\ in the formalizable approach: we start with some plausible set of generators, then after we have relations we check that all of the remaining primes up to the Minkowski bound are covered by the known classes. 
This improves the speed in the linear algebra, because we can get organized in the class group doing linear algebra with smaller matrices. 

In the conditional approach, we are guaranteed a (reasonably small) set of generators. (In a possible heuristic approach, we do not care.) \Magma\ previously incorrectly used a heuristic Euler product evaluation, but now we use a rigorous one and we can stop (in proof level {\tt GRH}) once we pass the analytic check explained in the previous section. 
If only the bounds for the generators of the class group are set to {\tt GRH}, a check for saturation 
of the class group is done.
%Really? So no saturation in the conditional case? Or only unit saturation? 

\subsection*{Saturating the class group}

What remains then is a formalizable algorithm for the class group. We again do this with saturation. 
The basic idea is as follows: if $\fraka$ is given and $\fraka^n = (a)$ is known to be principal, then we have
\begin{itemize}
\item
either $\fraka$ is not principal, or
\item
there exists $u \in \calO_K^\times$ such that $u a \in K^{\times n}$ is an $n$th power (in $\calO_K$). (In this instance, an $n$-th root will generate $\fraka$.)
\end{itemize}
As usual, we reduce to the case where $n=\ell$ is prime (by raising to the power $n/\ell$). We compute a GRH-conditional unit group. The implemented check will confirm that
the class group and the unit group are in fact $\ell$-saturated or terminate with an error message. 

We have to saturate the entire class group and not just a cyclic subgroup. Therefore, the approach is a bit more elaborate.

We use the notation $\Cl(\calO_K) \cong C_1 \times \cdots \times C_k$ to decompose the class group into a product of cyclic groups. For each factor $C_i$ we chose an ideal $\fraka_i$ representing a generator. Now, the function {\tt ClassGroupCyclicFactorGenerators} returns a generator $g_i$ for each principal ideal $\fraka_i^{\#C_i}$. Further, we denote by $u_0,\ldots,u_r$ generators of the unit group. 
We assume that $u_0$ is the only generator of finite order.

Assuming the prime $\ell$ divides $\#C_1,\ldots,\#C_j$ then the class and the unit group 
are $\ell$-saturated if and only if the subgroup $\langle u_0,\dots,u_r,g_1,\dots,g_j \rangle$ does not contain an $\ell$th power aside from $1$.  If $\ell \nmid \ord(u_0)$, then we can remove it from the subgroup.  
%
%the following set does not contain an $\ell$th power aside from 1:
%\begin{equation}
%\begin{cases}
%\{u_0^{e_0} \cdots u_r^{e_r} g_1^{f_1} \cdots g_j^{f_j} : e_0,\ldots,f_j \in \{0,\ldots,\ell-1\}\}, 
%& \textup{ if $\ell \mid \ord(u_0)$; } \\
%\{u_1^{e_1} \cdots u_r^{e_r} g_1^{f_1} \cdots g_j^{f_j} : e_1,\ldots,f_j \in \{0,\ldots,\ell-1\}\},
%& \textup{ if $\ell \nmid \ord(u_0)$.} 
%\end{cases}
%\end{equation}
The search for $\ell$-th powers is done in the
same way as the saturation of the unit group by
a reduction modulo $\frakp$ approach and linear
algebra over $\bbZ / \ell \bbZ$.

\begin{rmk}
The current C-level implementation of the class group saturation is switched off and it is no longer compatible with the current C code.
%turning it on  results in crashes. 
Therefore, this part is reimplemented in package code. 

As the above computation can be done modulo $h$-th powers, this allows to replace {\tt ClassGroupCyclicFactorGenerators} by a more efficient function that gives the results of the
linear algebra steps only modulo $h$ (resp. the final result modulo $h$-th powers). Further, it is not necessary to be compatible with the class group map, as it is not used in the saturation algorithm. As the computation is done by reduction modulo $\frakp$, all elements are represented as power products and multiplied out after reduction modulo $\frakp$.
\end{rmk}

Even in the case of imaginary quadratic fields with large discriminants, the time spent on discrete logarithms is negligible. Other types of fields result in much smaller class groups and the discrete logarithms are trivial to compute. The only step whose computation time is not negligible is the computation of the Smith form (with transformation matrix) of the 
relation matrix.

%\begin{ques}
%Using the Hermite normal form, you can move to a much smaller group by taking only the ideals starting from the first non-trivial diagonal element (sometimes called the \emph{essential} element). Are we doing this?
%\end{ques} 

% This is outdated with my better understanding as
% described above
%
%In terms of complexity, one has to inspect all the %cyclic order $\ell$ subgroups of the hypthetical class %group found; for each one has to pick a generator and %run the above computation. By the Cohen--Lenstra--%Martinet heuristics (``close to cyclic'', see heuristic %generators above), we have to check only very few %ideals in typical cases.
%
%\begin{rmk}
%In maliciously constructed cases, this can get slow and %a different approach is needed!
%\end{rmk}

The above saturation of the class group is implemented in KANT, but unfortunately, it was switched off. 
% (in V?). 
As of V2.29, a new implementation is available.
% it is switched back on again. 

% \section{Long example}

% Here is a run to compute the unit group formalizably. (We should label the steps and add the class group!)

% %\input{long_example.tex}

\section{Conclusions}  \label{sec:conclude}

\subsection*{Summary} 

To summarize, as of V2.29, the algorithms run as follows.
\begin{itemize}
\item 
To compute a formalizable class group, we use a heuristic set of generators and a heuristic evaluation of the Euler product, compute a set of relations until the analytic class number matches, and then (using exact methods) saturate the class group. Finally it is checked that the heuristically chosen set of generators is in fact a generating set for the class group.
\item 
To compute a rigorous unit group, we saturate using the regulator of the known subgroup.  (This could be made formalizable with ball arithmetic.)  
\item 
We have a GRH-conditional formalizable algorithm for the class group, but again only GRH-conditional rigorous algorithm for the unit group.  For a GRH-conditional rigorous algorithm, we use the GRH-bound, evaluate the analytic class number, and declare we are done with relations when we match.

\item We don't yet have a fast heuristic algorithm.
\end{itemize}

The computation of $\res_{s=1} \zeta_K(s)$ and error estimates are implemented in C.  
% It can be called directly via {\tt ResidueGRH, ResidueGRHbound}.
The saturation proof for the class group is implemented in package level.  
% \item
The description of the proof levels {\tt GRH, UserBound, Subgroup, Full} in the handbook has been updated.
% \item
By an early abort of the relation search (before the Euler product matches), we can test the saturation algorithms.
% \end{enumerate}

\subsection*{Benchmarks}
The table below shows the timinings for the steps performed.
\medskip

\begin{center}
\begin{tabular}{c||c|c|c|c}
               & max order   &                &                   &  \\ 
Polynomial     & Euler prod & $\Cl(\calO_K)$ &  $\calO_K^\times$ & $\Cl(\calO_K)$   \\
               & residue    &                &                   & saturation  \\
\hline\hline
\rule{0pt}{2.6ex} $x^2 + 3$  & 0.000 & 0.030 & 0.000 & 0.000 \\ 
$x^2 + 10^{30} + 57$  & 0.010 & 2.090 & 0.000 & 5.620 \\ 
$x^2 - 10^{30} - 57$  & 0.010 & 0.960 & 2.010 & 0.000 \\ 
$x^3 - 17$  & 0.010 & 0.020 & 0.000 & 0.000 \\ 
$x^3 - x + 2015993900449$  & 0.010 & 0.450 & 1.800 & 0.200 \\ 
$x^6 + 4x^3 + 30x^2 + 4$  & 0.020 & 0.070 & 0.010 & 0.010 \\ 
$x^8 + 4346x^4 + 169$  & 0.050 & 0.370 & 0.280 & 0.070 \\ 
$x^8 + 8932x^4 + 17161$  & 0.050 & 0.640 & 1.650 & 0.120 \\ 
$x^{12} + 4x^6 + 270x^4 + 4$  & 0.080 & 0.640 & 0.130 & 0.110 \\ 
$x^{12} + 422x^6 + 5070x^4 + 44521$  & 0.130 & 3.410 & 18.820 & 0.720 \\ 
\end{tabular}
\end{center}
\medskip

The time to saturate the unit group is proprtional to the upper index bound and the time 
to fully check the class group is proportional to the Minkowski bound. 
Thus, both is only possible in small examples.
\medskip

\begin{center}
\begin{tabular}{c||c|c|c}
Polynomial     & Conditional & Verify         & $\calO_K^\times$    \\
               & results    &  $\Cl(\calO_K)$& saturation  \\
\hline\hline
\rule{0pt}{2.6ex} $ x^2 - 13 $ & 0.010 & 0.000 & 0.000 \\ 
$ x^2 - 113 $ & 0.030 & 0.000 & 0.000 \\ 
$ x^2 - 1013 $ & 0.010 & 0.000 & 0.000 \\ 
$ x^2 - 10103 $ & 0.010 & 0.000 & 0.000 \\ 
$ x^2 - 10^5-103 $ & 0.020 & 0.000 & 0.000 \\ 
$ x^2 - 10^6-1003 $ & 0.030 & 0.000 & 0.000 \\ 
$ x^2 - 10^7-1009 $ & 0.030 & 0.000 & 0.010 \\ 
$ x^2 - 10^8 -10017 $ & 0.030 & 0.020 & 0.020 \\ 
$ x^2 - 10^9-10029 $ & 0.040 & 0.050 & 0.070 \\ 
$ x^2 - 10^{10}-100003 $ & 0.090 & 0.850 & 0.280 \\ 
$ x^2 - 10^{11}-100009 $ & 0.080 & 1.270 & 0.560 \\ 
$ x^2 - 10^{12}-1000021 $ & 0.180 & 10.190 & 2.620 \\ 
$ x^2 - 10^{13}-1000029 $ & 0.450 & 81.710 & 8.580 \\ 
$ x^2 - 10^{14}-10000003 $ & 1.990 & --- & 89.170 
\end{tabular}
\end{center}

\medskip
%Thus, the rigorous approach is only practical in small examples. 
The next larger example was terminated when using more than 20 GB of memory.

\subsection*{Testing}
Among other test, the following was performed with repeated runs of more than 10000 number fields. The relation search was terminated early, once the relation matrix and the unit group reached full rank, but  without a match in the analytic class number formula. 

The resulting class groups and unit groups were compared with GRH-conditional results. In case the results did not coincide, a rigorous computation was requestet. This called the proof algorithms, as the result was viewed as conditional. The test found:
\begin{itemize}
\item
Corrected indices of unit groups (by a search for small units) range from 2 to 648 and include various primes such as 2, 3, 7, 11, 13, 17, 23.
\item
Detected indices of unit groups (resulting in error messages) range from 2 to 648 (including 101 and other odd primes).
\item
Detected indices of class groups (resulting in error messages) range from 2 to 24. Prime factors other than 2 and 3 did not show up.
\end{itemize}
In total an incorrect class group occurred about 4000 times, an incorrect unit group occurred about 9000.
For further improvements of the algorithm we can learn from this test that a full rank relation matrix will give the correct class group with higher probability than the unit group. An incorrect class group will most likely be off by a factor of the shape $2^i 3^j$.
Further, the unit group derived from an incomplete relation matrix can have large index in the full unit group. In case the unit group is incorrect, the index is 2 with a probability of about $2/3$.

\subsection*{Possible future work} 

\begin{enumerate}
\item The algorithm to generate relations used in the class group computation would be useful in other contexts. Could we have access to these functions at the user level, perhaps with some parameters to control the factor base and how many relations to generate?  This would be potentially useful in Diophantine contexts, e.g. in the Brauer--Manin obstruction. For this application it is not needed that the kernel of the relation matrix gives the whole unit group. 
More precisely, we would like to find elements whose norm is an $\ell$th power, without necessarily completeness results; relation finding code provides an important input for that.
\item It would be useful to be able to work with unexpanded power products for ideals and units as a well-defined type.
\item For abelian fields, can we do better in class group computations: for example, computing class numbers using Bernoulli numbers?
\item Is there any potential speedup to compute just $(\Cl \calO_K)/(\Cl \calO_K)^{\ell}$?  In principle, the linear algebra is faster (no awfulness with Smith form), if we have generators we only need to do $\ell$-saturation which is easier.  The case $\ell=2$ shows up frequently in applications.
\item Are there other heuristics or computational observations that would lead to faster heuristic algorithms?  Sometimes one just wants to know the answer as quickly as possible without rigor, then after the experimental phase is over certifying it.
\end{enumerate}


\begin{thebibliography}{BCP97}

\bibitem[B94]{bach_94}
Eric Bach, \emph{Improved approximations for Euler products}, 
in {\it Number theory (Halifax, NS, 1994)}, 13--28, 
CMS Conf. Proc., 15, Amer. Math. Soc., Providence, RI.


\bibitem[B96]{Bach}
Eric Bach, \emph{Explicit bounds for primality testing and related problems}, 
Math. Comp. {\bf 55} (1990), no.~191, 355--380.


\bibitem[BDF08]{MR2373197}
Karim Belabas, Francisco Diaz~y~Diaz and Eduardo Friedman, 
\emph{Small generators of the ideal class group}, 
Math. Comp. {\bf 77} (2008), no.~262, 1185--1197.


\bibitem[BF15]{MR3266965}
Karim Belabas and Eduardo Friedman, \emph{Computing the residue of the Dedekind zeta function}, 
Math.\ Comp.\ {\bf 84} (2015), no.~291, 357--369.

\bibitem[BF14]{BF}
Jean-Fran\c{c}ois Biasse and Claus Fieker, \emph{Subexponential class group and unit group computation in large degree number fields}, LMS J.\ Comput.\ Math., \textbf{17}(A) (2014), 385--403.

\bibitem[BCP97]{Magma}
Wieb Bosma, John Cannon, and Catherine Playoust, \emph{The Magma algebra system. I. The user language}, J.~Symbolic Comput.\ \textbf{24} (1997), 235--265.

\bibitem[Co93]{Cohen}
Henri Cohen, {\it A course in computational algebraic number theory}, Grad. Texts in Math., vol.~138, Springer, Berlin, 1993.

\bibitem[Co00]{Cohen2}
Henri Cohen, \emph{Advanced topics in computational number theory}, Grad. Texts in Math., vol.~193, Springer-Verlag, New York, 2000.

\bibitem[D04]{D04}
Tim Dokchitser, 
\emph{Computing special values of motivic $L$-functions}, 
Experiment.~ Math.\ {\bf 13} (2004), no.~2, 137--149.

\bibitem[Duk03]{Duke}
W.\ Duke, \emph{Extreme values of Artin $L$-functions and class numbers}, Compositio Math.\ \textbf{136} (2003), 103--115.

\bibitem[F01]{F01}
Claus Fieker, \emph{Computing class fields via the Artin map}, Math. Comp. {\bf 70} (2001), no.~235, 1293--1303.

\bibitem[F06]{F06}
Claus Fieker, \emph{Applications of the class field theory of global fields}, Discovering mathematics with Magma, Algorithms Comput. Math., vol.~19, Springer, Berlin, 2006, 31--62.

\bibitem[GL22]{Garcia}
Stephan Ramon Garcia and Ethan Simpson Lee, 
\emph{Unconditional explicit Mertens' theorems for number fields and Dedekind zeta residue bounds}, 
Ramanujan J. {\bf 57} (2022), no.~3, 1169--1191.

\bibitem[GS03]{GS03}
A.~J. Granville and K.~Soundararajan, \emph{The distribution of values of $L(1,\chi_d)$}, Geom. Funct. Anal. {\bf 13} (2003), no.~5, 992--1028.

\bibitem[L92]{L92}
H.~W. Lenstra Jr., \emph{Algorithms in algebraic number theory}, Bull. Amer. Math. Soc. (N.S.) {\bf 26} (1992), no.~2, 211--244

\bibitem[SMC06]{SMC06}
J\'ozsef\ S\'andor, Dragoslav~S.\ Mitrinovi\'c,and Borislav\ Crstici, {\it Handbook of number theory.~I}, Springer, Dordrecht, 2006.

\end{thebibliography}
\end{document}